\documentclass[12pt]{amsart}
\usepackage{hyperref}
\hypersetup{nesting=true,debug=true,naturalnames=true}
\usepackage{graphicx,amssymb,upref}
\usepackage{xcolor}
\title{Revisiting the Hales--Jewett Theorem }
\author{Arpita Ghosh}
\date{}
\curraddr{Department of Mathematics, University of Haifa, Israel}
\email{arpi.arpi16@gmail.com}



\newtheorem{theorem}{Theorem}
\newtheorem{lemma}{Lemma}
\newtheorem{prop}{Proposition}
\newtheorem{definition}{Definition}

\begin{document}

\begin{abstract}
This short note establishes an abstract Hales--Jewett theorem for semigroups equipped with a finite family of retractions. The proof relies on the interplay between retractions and tensor products of ultrafilters.
\end{abstract}

\maketitle

\section{Introduction}

Ramsey theory is founded on the principle that sufficiently large or sufficiently rich structures inevitably contain highly organized substructures. Among its deepest and most influential results is the Hales--Jewett theorem, established by Hales and Jewett in 1963~\cite{HalesJewett}. In its classical form, the theorem asserts that for every pair of positive integers $r$ and $n$, there exists an integer $N$ such that whenever the set of all words of length $N$ over an $n$-letter alphabet is colored with $r$ colors, one can find a monochromatic combinatorial line. This theorem is widely regarded as one of the cornerstones of modern Ramsey theory, and it has had profound consequences in combinatorics, ergodic theory, topological dynamics, and number theory.

The significance of the Hales--Jewett theorem lies not only in its intrinsic combinatorial depth, but also in its role as a unifying principle. It simultaneously generalizes several earlier partition theorems, including van der Waerden's theorem on arithmetic progressions~\cite{vdW}. In particular, many phenomena concerning monochromatic configurations can be understood as manifestations of the combinatorial richness captured by the Hales--Jewett framework.

Over the years, numerous extensions and reformulations of the Hales--Jewett theorem have been developed. Notable among these are density
version of the Hales–Jewett theorem by Furstenberg and Katznelson \cite{FurstenbergKatznelson} which generalizes Szem\'{e}redi’s theorem \cite{Szeme}, the Graham--Rothschild parameter sets theorem~\cite{GrahamRothschild}, and a polynomial Hales–Jewett theorem that extends the polynomial van
der Waerden theorem by Bergelson and Leibman \cite{BerLeib}. The theorem also admits elegant proofs from several different perspectives, including combinatorial, ergodic, and topological approaches. A particularly fruitful viewpoint arises from the algebra of the Stone--\v{C}ech compactification of discrete semigroups, pioneered by Ellis and Furstenberg, and developed systematically by Bergelson, Hindman, Strauss, Beiglb\"{o}ck \cite{Bergelson1986,HindmanStrauss,Beiglbock}.

In \cite{Kop}, Koppelberg introduced and proved an abstract version of Hales--Jewett theorem, and as an easy consequence,  obtained the classical Hales--Jewett theorem, van der Waerden’s theorem, and Gallai’s theorem as special cases.

The purpose of this paper is to provide an entirely algebraic proof of the abstract Hales--Jewett theorem in the setting of semigroups equipped with a finite family of retractions. Our proof relies on a careful analysis of the interaction between retractions and tensor products of ultrafilters. More precisely, we work in the following setting.

\begin{definition}
Let $(S,\bullet)$ be a semigroup, and let $T$ be a subsemigroup of $S$.
\begin{enumerate}
    \item We say that $T$ is a \emph{nice subsemigroup} of $S$ if the complement $R=S\setminus T$
    is a two-sided ideal of $S$. Equivalently, whenever $T\neq S$, one has
    $x\bullet y\in T
    \quad\Longleftrightarrow\quad$
    the product of two elements $x \bullet y$ of $S$ is in $T$ if and only if both $x$ and $y$ are in $T$.

    \item A \emph{retraction} from $S$ onto $T$ is a semigroup homomorphism
    $\sigma\colon S\to T$
    such that $\sigma$ restricts to the identity on $T$; that is, $
    \sigma(t)=t  \text{ for all } t\in T.$
\end{enumerate}
\end{definition}

For  $\Sigma$ be a finite family of retractions from $S$ onto $T$, we show that for every finite coloring of $T$, there exists an element $v\in S\setminus T$ such that the set $\{\sigma(v):\sigma\in\Sigma\}$ is monochromatic. This formulation recovers the classical Hales--Jewett theorem when $S$ is the semigroup of variable words, $T$ is the semigroup of constant words, and the retractions are the substitution maps obtained by replacing the variable by letters from the alphabet.

\vspace{0.5cm}
\noindent\textbf{Notation.}
\begin{enumerate} 
    \item For any semigroup homomorphism $f\colon S\to T$, we denote by
    $\widetilde{f}\colon \beta S\to \beta T$
    its continuous extension to the Stone--Čech compactifications.

    \item For $n\in \mathbb{N}$ and $\mathcal{U}\in \beta S$, we write
    \(
    \mathcal{U}^{\otimes n}
    =
    \underbrace{\mathcal{U}\otimes \cdots \otimes \mathcal{U}}_{n\text{ times}}
    \)
    for the $n$-fold tensor product of $\mathcal{U}$.

    \item For $n\in \mathbb{N}$ and $\mathcal{U}\in \beta S$, we write
    \(
    \mathcal{U}^{\bullet n}
    =
    \underbrace{\mathcal{U}\bullet \cdots \bullet \mathcal{U}}_{n\text{ times}}
    \)
    for the $n$-fold product of $\mathcal{U}$ in $\beta S$.
\end{enumerate}

\section{Preliminaries}

In this section we briefly recall the basic notions concerning ultrafilters, the Stone--\v{C}ech compactification, and tensor products of ultrafilters that will be used throughout the paper. Standard references for this material include~\cite{HindmanStrauss}.

An \emph{ultrafilter} on a nonempty set $X$ is a maximal filter on $X$; equivalently, it is a collection $\mathcal U$ of subsets of $X$ such that for every $A\subseteq X$, exactly one of $A$ and $X\setminus A$ belongs to $\mathcal U$. 

If $\mathcal U$ is an ultrafilter on $I$ and $f:I\to J$ is a map, then the \emph{image ultrafilter} $f(\mathcal U)$ on $J$ is given by
\[
A\in f(\mathcal U)\quad\Longleftrightarrow\quad f^{-1}(A)\in\mathcal U
\]
for all $A\subseteq J$.

For a discrete semigroup $S$, the \emph{Stone--\v{C}ech compactification} $\beta S$ is the set of all ultrafilters on $S$. Identifying each $s\in S$ with the principal ultrafilter at $s$, we regard $S$ as a dense subset of $\beta S$. For each $A\subseteq S$, the set
\[
\overline{A}=\{\mathcal U\in\beta S:A\in\mathcal U\}
\]
is clopen, and these sets form a basis for the compact Hausdorff topology on $\beta S$.

The semigroup operation on $S$ extends uniquely to a right-topological semigroup operation on $\beta S$. Explicitly, if $\mathcal{U},\mathcal{V}\in\beta S$ and $A\subseteq S$, then
\[
A\in\mathcal{U}\bullet\mathcal{V}
\quad\Longleftrightarrow\quad
\{s\in S:s^{-1}\bullet A\in\mathcal{V}\}\in\mathcal{U},
\]
where
\[
s^{-1}\bullet A=\{t\in S:s\bullet t\in A\}.
\]
With this operation, $(\beta S,\bullet)$ is a compact right-topological semigroup.

Let $I$ and $J$ be nonempty sets, and let $\mathcal{U}\in\beta I$ and $\mathcal{V}\in\beta J$. The \emph{tensor product} of $\mathcal{U}$ and $\mathcal{V}$ is the ultrafilter $\mathcal{U}\otimes\mathcal{V}$ on $I\times J$ defined by the rule
\[
X\in\mathcal{U}\otimes\mathcal{V}
\quad\Longleftrightarrow\quad
\{i\in I:\{j\in J:(i,j)\in X\}\in\mathcal{V}\}\in\mathcal{U}
\]
for every $X\subseteq I\times J$.

Equivalently, a subset $X\subseteq I\times J$ belongs to $\mathcal{U}\otimes\mathcal{V}$ if and only if, for $\mathcal{U}$-many $i\in I$, the vertical section
\[
X_i=\{j\in J:(i,j)\in X\}
\]
belongs to $\mathcal{V}$.

The tensor product is associative in the natural sense and provides a convenient mechanism for handling higher-dimensional combinatorial configurations.

\begin{prop}\label{pindultra}
Let $\sigma\in \Sigma$, and let $\Psi_\sigma:S^k\to T$ be the associated homomorphism given by $\Psi_\sigma(v_1, \cdots , v_k)= \sigma(v_1\bullet \cdots \bullet v_k)$. Then the continuous extension $\widetilde{\Psi_\sigma}:\beta(S^k)\to \beta T$
satisfies
\[
\widetilde{\Psi_\sigma}\bigl(\mathcal{V}^{\otimes k}\bigr)
   =\bigl(\widetilde{\sigma}(\mathcal{V})\bigr)^{\bullet k}
\]
for every ultrafilter $\mathcal{V}\in \beta S$, where $\widetilde{\sigma}$ is the continuous extension of $\sigma: S \rightarrow T$.
\end{prop}

\begin{proof}
Note that it is enough to prove for $k=2$. Let $A\subseteq W$. By definition,
$
A\in \widetilde{\Psi_\sigma}\bigl(\mathcal{V}^{\otimes 2}\bigr)
$
if and only if $\Psi_\sigma^{-1}(A)\in \mathcal{V}^{\otimes 2}.
$ By the definition of $\Psi_\sigma$ we have 
\[
\Psi_\sigma^{-1}(A)
=
\bigl\{(v_1, v_2)\in S^2:
\sigma(v_1)\bullet \sigma(v_2)\in A\bigr\} \in \mathcal{V}^{\otimes 2}.
\]
This implies that 
\[
\{v_1 \in S: \{v_2 \in S: \sigma(v_1)\bullet \sigma(v_2) \in A\} \in \mathcal{V}\} \in \mathcal{V}.
\]
On the other hand, by the definition of the product ultrafilter,
\[
A\in \bigl(\widetilde{\sigma}(\mathcal{V})\bigr)^{\bullet 2}
\iff
\bigl\{w_1\in T:
\{w_2\in T: w_1\bullet w_2\in A\}\in \widetilde{\sigma}(\mathcal{V})
\bigr\}\in \widetilde{\sigma}(\mathcal{V}).
\]
Unwinding the definition of $\widetilde{\sigma}(\mathcal{V})$, this is equivalent to
\[
\bigl\{w_1\in T:
\{v_2\in S: w_1\bullet \sigma(v_2)\in A\}\in \mathcal{V}
\bigr\}\in \widetilde{\sigma}(\mathcal{V}),
\]
and hence, again by the definition of $\widetilde{\sigma}(\mathcal{V})$,
\[
\bigl\{v_1\in S:
\{v_2\in S: \sigma(v_1)\bullet \sigma(v_2)\in A\}\in \mathcal{V}
\bigr\}\in \mathcal{V}.
\]

Hence
\[
\widetilde{\Psi_\sigma}\bigl(\mathcal{V}^{\otimes k}\bigr)
=
\bigl(\widetilde{\sigma}(\mathcal{V})\bigr)^k,
\]
as required.
\end{proof}

\section{The Proof of the main Theorem}
Let $(S,\bullet)$ be a semigroup and $T \subseteq S$ a nice subsemigroup, i.e., $R = S \setminus T$ is a two-sided ideal. Let $\Sigma$ be a finite family of retractions $\sigma: S \to T$. Let $\widetilde{\sigma}: \beta S \rightarrow \beta T$ be the continuous extension of $\sigma.$

\begin{lemma}\label{Lemma2}
The following statements are equivalent:

\begin{itemize}
    \item[(a)] Assume that $T$ is a nice subsemigroup of $S,$ and $\Sigma$ is a finite set of retractions from $S$ to $T.$ Moreover assume that $T=B_1 \cup \cdots \cup B_r$ is a coloring of $T$ with finitely many colors. Then there is some $v \in R=S \setminus T$ such that $\{\sigma(v) : \sigma \in \Sigma\}$ is monochromatic. 

   \item [(b)]There exists an ultrafilter $\mathcal{U}$ on $S$ such that
\[
\widetilde{\sigma}(\mathcal{U}) =\widetilde{\tau}(\mathcal{U}) \quad \forall \sigma,\tau \in \Sigma.
\]

\end{itemize}
\end{lemma}

\begin{proof}
$(b)\Rightarrow (a)$. Let us consider a finite coloring of $T$ as follows:
\[
T=B_1\cup\cdots\cup B_r.
\]
 Choose $i\in\{1,\dots,r\}$ such that
$B:=B_i\in \widetilde{\sigma}(\mathcal U).$ Since $\widetilde{\sigma}(\mathcal U)=\widetilde{\tau}(\mathcal U)
\mbox{ for all }\sigma,\tau\in\Sigma,$ it follows that
\[
\sigma^{-1}(B)\in\mathcal U
\qquad\text{for every }\sigma\in\Sigma.
\]
As $\Sigma$ is finite, we therefore have $\bigcap_{\sigma\in\Sigma}\sigma^{-1}(B)\in\mathcal U.$ Since $R$ is a two-sided ideal of $S$, its closure $\overline{R}:=\operatorname{cl}_{\beta S}(R)$ is a left ideal of $\beta S$. Because $\beta S\cdot\mathcal U$ is a minimal left ideal and
\[
\beta S\cdot\mathcal U\subseteq \overline{R},
\]
it follows in particular that $\mathcal U\in\overline{R}.$ Hence $R\in\mathcal U$. Consequently,
\[
R\cap\bigcap_{\sigma\in\Sigma}\sigma^{-1}(B)\in\mathcal U,
\]
so this set is nonempty. Choose
\[
v\in R\cap\bigcap_{\sigma\in\Sigma}\sigma^{-1}(B).
\]
Then $v\in R$, and for every $\sigma\in\Sigma$ we have $\sigma(v)\in B$. Therefore,
\[
\{\sigma(v):\sigma\in\Sigma\}\subseteq B=B_i,
\]
as required.

$ (a) \Rightarrow (b)$. For each $A \subseteq T$, define
\[
X_A = \{v \in S : \{\sigma(v): \sigma \in \Sigma\} \subseteq A \text{ or }  \{\sigma(v): \sigma \in \Sigma\} \subseteq A^c\}.
\]

We first show that the family $\mathcal{G}=\{X_A: A \subseteq T\}$ has the finite intersection property. Let $A_1,\dots,A_n \subseteq T$. 


For each function $\chi : \{1, \cdots ,n\} \rightarrow \{+,-\}$ define $C_{\chi} = \bigcap_{i=1}^n A_i^{\chi(i)}$  where we understand $A^+=A$ and $A^{-}=A^c.$ Consider a finite partition of  $W= \bigcup_{\chi}C_{\chi}$. Then by $(a)$, there exist $v \in R$ and $\chi$ such that $$\{\sigma(v) : \sigma \in \Sigma\} \subseteq C_{\chi}.$$ Hence for each $i$, $\{\sigma(v) : \sigma \in \Sigma\} \in A_i$ or $\{\sigma(v) : \sigma \in \Sigma\} \in A_i^c$. This implies that $v \in X_{A_i}$ for all $i,$ therefore $v \in \bigcap_{i=1}^nX_{A_i}.$ Thus $\mathcal{G}$ has finite intersection property, so there exists an ultrafilter on $\mathcal{U}$.

We now show that all images coincide. Suppose for contradiction that there exist $\sigma,\tau \in \Sigma$ and $A \subseteq T$ such that
\[
A \in \widetilde{\sigma}(\mathcal{U}) \quad \text{and} \quad A^c \in \widetilde{\tau}(\mathcal{U}).
\]
Then
\[
\sigma^{-1}(A) \in \mathcal{U}, \quad \tau^{-1}(A^c) \in \mathcal{U}.
\]
Hence
\[
Y := \sigma^{-1}(A) \cap \tau^{-1}(A^c) \in \mathcal{U}.
\]

But for every $v \in Y$, we have $\sigma(v) \in A$ and $\tau(v) \in A^c$, so
\[
v \notin X_A.
\]
Thus $Y \cap X(A) = \emptyset$, contradicting that both sets belong to $\mathcal{U}$. Therefore,
\[
\widetilde{\sigma}(\mathcal{U}) = \widetilde{\tau}(\mathcal{U}) \quad \forall \sigma,\tau \in \Sigma.
\]
Hence the result follows.
\end{proof}

\begin{theorem}[Hales--Jewett Theorem]
Assume that $T$ is a nice subsemigroup of $S,$ and $\Sigma$ is a finite set of retractions from $S$ to $T.$ Moreover assume that $T=B_1 \cup \cdots \cup B_r$ is a coloring of $T$ with finitely many colors. Then there is some $v \in R=S \setminus T$ such that $\{\sigma(v) : \sigma \in \Sigma\}$ is monochromatic.
\end{theorem}

\begin{proof}
We will use induction method on cardinality of $\Sigma.$ For $|\Sigma|=1$, this is trivially true. Let us assume that $|\Sigma|=k$ with $\Sigma= \{\sigma_1, \cdots, \sigma_k\}$. 
Let $\mathcal{U}$ be an ultrafiler of $S$ as in the Lemma \ref{Lemma2} and by induction hypothesis set $\mathcal{V} := \widetilde{\sigma}(\mathcal{U})$ for any $\sigma \in \Sigma \setminus  \{\sigma_k\}$. Now consider $\mathcal{W }= \widetilde{\sigma_k}(\mathcal{U}).$

Define ultrafilters
\[
\mathcal{Z}_{i} := \mathcal{W}^{\bullet i} \bullet \mathcal{V}^{\bullet (r+2-i)}, \quad i=1,\dots,r+1.
\]

Each $\mathcal{Z}_i$ contains one of the colors $B_1,\dots,B_r$. Since there are $r+1$ ultrafilters, there exist $p<q$ and $j$ such that
\[
B_j \in \mathcal{Z}_p \cap \mathcal{Z}_q.
\]

Define a set
\[
\Gamma := \{ v \in T : v^{-1} \bullet B_j \in \mathcal{V}^{\bullet (r+2-q)} \} 
.\]
Here $v^{-1} \bullet B_j= \{w \in T : v\bullet w \in B_j\}.$
Since $B_j \in \mathcal{Z}_p$ and $B_j \in \mathcal{Z}_q$, it follows that
\[
\Gamma \in \mathcal{W}^{\bullet p} \bullet \mathcal{V}^{\bullet (q-p)} \quad \text{and} \quad
\Gamma \in \mathcal{W}^{\bullet q} = \mathcal{W}^{\bullet p} \bullet \mathcal{W}^{\bullet (q-p)}.
\]

Define
\[
\Gamma_1 := \{ u \in T : u^{-1} \bullet \Gamma \in \mathcal{V}^{\bullet (q-p)} \},
\]
\[
\Gamma_2 := \{ u \in T : u^{-1}\bullet \Gamma \in \mathcal{W}^{\bullet (q-p)} \}.
\]
Then $\Gamma_1, \Gamma_2 \in \mathcal{W}^{\bullet p}$.
Hence $\Gamma_1\cap\Gamma_2\in\mathcal{W}^{\bullet p},$ so this set is nonempty. Choose
\[
u\in \Gamma_1\cap\Gamma_2.
\]
Therefore, we have $u^{-1} \bullet \Gamma \in \mathcal{V}^{\bullet (q-p)} \cap \mathcal{W}^{\bullet (q-p)}.$

Let $m = q-p$ and consider the tensor product ultrafilter $\mathcal{U}^{\otimes m}$ on $S^m$. For each $\sigma \in \Sigma$, define a map $\Psi_\sigma : S^m \to T$ given by the following formula:
\[
\Psi_\sigma(v_1,\dots,v_m) := \sigma(v_1 \bullet \cdots \bullet v_m).
\]
Then by Proposition \ref{pindultra},

\[
\widetilde{\Psi}_\sigma(\mathcal{U}^{\otimes m})
=
\begin{cases}
\widetilde{\sigma}(\mathcal{U})^{\bullet m}
=\mathcal{V}^{\bullet m},
& \text{if } \sigma\in \Sigma\setminus\{\sigma_k\}.\\ \widetilde{\sigma_k}(\mathcal{U})^{\bullet m} = \mathcal{W}^{\bullet m} & \text{if } \sigma=\sigma_k.
\end{cases}
\]

Hence
\[
\bigcap_{\sigma \in \Sigma} \Psi_\sigma^{-1}(u^{-1} \bullet \Gamma) \in \mathcal{U}^{\otimes m}.
\]
Choose $(v_1,\dots,v_m)$ in the above intersection and we have $\Psi_{\sigma}(v_1, \cdots , v_m) \in u^{-1} \bullet \Gamma.$ Now set $w = v_1 \bullet \cdots \bullet v_m$. Then for every $\sigma \in \Sigma$, $\sigma(w) \in u^{-1} \bullet \Gamma,$
so
\[
u \bullet \sigma(w) \in \Gamma.
\]
For each $\sigma\in\Sigma$, the inclusion $u \bullet \sigma(w)\in\Gamma$ implies that $(u \bullet \sigma(w))^{-1} \bullet B_j\in\mathcal{V}^{\bullet(r+2-q)}.
$
Since $\Sigma$ is finite, it follows that
\[
\bigcap_{\sigma\in\Sigma}(u \bullet \sigma(w))^{-1} \bullet B_j
\in\mathcal{V}^{\bullet(r+2-q)}.
\]
Hence this set is nonempty, and we may choose
\[
t\in \bigcap_{\sigma\in\Sigma}(u \bullet \sigma(w))^{-1} \bullet B_j.
\]

Now define $v:=u \bullet w \bullet t.$ Since $u,t\in T$ and each $\sigma\in\Sigma$ is a retraction onto $T$, we have 
\[
\sigma(v)=\sigma(u) \bullet \sigma(w) \bullet \sigma(t)=u \bullet \sigma(w) \bullet t\in B_j
\]
for every $\sigma\in\Sigma$. Finally, since $R = S \setminus T$ is a two-sided ideal, we may ensure that $v \in R$.
\end{proof}


\begin{thebibliography}{99}

\bibitem{Beiglbock}
M. ~Beiglb\"{o}ck,
\newblock A variant of the Hales-Jewett theorem.
\newblock \emph{Bull. Lond. Math. Soc. 
} \textbf{40} (2008), no. 2, 210–216.

\bibitem{Bergelson1986}
V.~Bergelson,
\newblock Ergodic Ramsey theory---an update,
\newblock \emph{Ergodic Theory of $\mathbb{Z}^d$ Actions (Warwick, 1993--1994)},
London Math. Soc. Lecture Note Ser. \textbf{228}, Cambridge Univ. Press, Cambridge, 1996, pp.~1--61.

\bibitem{BerLeib}
V. Bergelson and A. Leibman, 
\newblock Polynomial extensions of van der Waerden’s and Szemer\'edi’s theorems,
\newblock \emph{J. Amer. Math. Soc.} \textbf{9} (1996) 725–753.

\bibitem{FurstenbergKatznelson}
H.~Furstenberg and Y.~Katznelson,
\newblock A density version of the Hales--Jewett theorem,
\newblock \emph{J. Anal. Math.} \textbf{57} (1991), 64--119.

\bibitem{GrahamRothschild}
R.~L.~Graham and B.~L.~Rothschild,
\newblock Ramsey's theorem for $n$-parameter sets,
\newblock \emph{Trans. Amer. Math. Soc.} \textbf{159} (1971), 257--292.

\bibitem{HalesJewett}
A.~W.~Hales and R.~I.~Jewett,
\newblock Regularity and positional games,
\newblock \emph{Trans. Amer. Math. Soc.} \textbf{106} (1963), 222--229.

\bibitem{HindmanStrauss}
N.~Hindman and D.~Strauss,
\newblock \emph{Algebra in the Stone--\v{C}ech Compactification: Theory and Applications},
2nd ed., de Gruyter, Berlin, 2012.

\bibitem{Kop}
S.~Koppelberg,
\newblock The Hales-Jewett theorem via retractions.
\newblock \emph{Topology Proc.} \textbf{28} (2004), no. 2, 595–601.

\bibitem{Szeme}
E. ~Szem\'{e}redi, 
\newblock On sets of integers containing no k elements in arithmetic progression,
\newblock \emph{Acta. Math.}  \textbf{27}
(1975) 199–245

\bibitem{vdW}
B.~L.~van der Waerden,
\newblock Beweis einer Baudetschen Vermutung,
\newblock \emph{Nieuw Arch. Wiskd.} \textbf{15} (1927), 212--216.



\end{thebibliography}
\end{document}